\documentclass[a4paper,11pt]{article}
\usepackage{amsmath}
\usepackage{amssymb}
\usepackage{amscd}
\usepackage{amsthm}

\usepackage{graphicx}
\usepackage{setspace}

\usepackage{fullpage}
\usepackage{bm}
\usepackage{amsthm}
\usepackage{pdfsync}
\newtheorem{theorem}{Theorem}[section]
\newtheorem{lemma}[theorem]{Lemma}

\newtheorem{corollary}[theorem]{Corollary}
\newtheorem{result}[theorem]{Result}
\theoremstyle{definition}

\newtheorem*{definition}{Definition}
\newtheorem*{corollary1}{Corollary \ref{subspreads}}
\newtheorem{remark}{Remark}

\newcommand{\FF}{\mathbb{F}}

\newcommand{\Fq}{\mathbb{F}_q}
\newcommand{\Fqn}{\mathbb{F}_{q^n}}

\newcommand{\Fqt}{\mathbb{F}_{q^2}}
\newcommand{\Fqk}{\mathbb{F}_{q^k}}

\newcommand{\B}{\mathcal B}

\newcommand{\D}{\mathcal D}
\newcommand{\C}{\mathcal C}

\newcommand{\R}{\mathcal R}

\def\A{\mathcal{A}}
\def\S{\mathcal{S}}
\def\P{\mathcal P}

\def\F{\mathbb{F}}
\def\Fq{{\mathbb{F}}_q}

\def\Aut{\mathrm{Aut}}

\def\PG{\mathrm{PG}}
\def\AG{\mathrm{AG}}

\def\PGL{\mathrm{PGL}}
\def\PGammaL{\mathrm{P\Gamma L}}

\def\N{\mathcal{N}}

\def\L{\mathcal{L}}

\newcommand{\npmatrix}[1]{\left( \begin{matrix} #1 \end{matrix} \right)}

\begin{document}
\title{Subgeometries in the Andr\'e/Bruck-Bose representation}
\author{Sara Rottey\thanks{Partially supported by Fonds Professor Frans Wuytack (UGent).}\and John Sheekey\thanks{Supported by the Research Foundation Flanders –- Belgium (FWO-Vlaanderen).} \and Geertrui Van de Voorde\footnotemark[2]}
\maketitle

\begin{abstract} We consider the Andr\'e/Bruck-Bose representation of the projective plane $\PG(2,q^n)$ in $\PG(2n,q)$. We investigate the representation of $\F_{q^k}$-sublines and $\F_{q^k}$-subplanes of $\PG(2,q^n)$, extending the results for $n=3$ of \cite{BarJack2} and correcting the general result of \cite{BarJack1}. We characterise the representation of $\F_{q^k}$-sublines tangent to or contained in the line at infinity, $\F_q$-sublines external to the line at infinity, $\F_q$-subplanes tangent to and $\F_{q^k}$-subplanes secant to the line at infinity.
\end{abstract}

\section{Preliminaries}
The {\em Andr\'e/Bruck-Bose representation}, or in short {\em ABB-representation}, is a well-known representation of the projective plane $\PG(2,q^n)$ in $\PG(2n,q)$ (for details, see Subsection \ref{ABB}). The ABB-representation of sublines, subplanes and unitals of $\PG(2,q^2)$ in $\PG(4,q)$ has been thoroughly studied (see \cite{Ebert}). In \cite{squinn} the authors extended these results and characterised $\F_{q^{2^i}}$-sublines tangent to and $\F_{q^{2^i}}$-subplanes secant to the line at infinity in the ABB-representation of $\PG(2,q^{2^n})$ in $\PG(2^{n+1},q)$.
Recently, the characterisation of $\F_q$-sublines and $\F_q$-subplanes tangent to the line at infinity in the ABB-representation of $\PG(2,q^3)$ in $\PG(6,q)$ was given in \cite{BarJack1} and \cite{BarJack2}.
In \cite{BarJack1}, the authors extended their proof for sublines with $n=3$ to general $n$. This generalisation is as we will see in Theorem \ref{externalsubline} not entirely correct (unless $n$ is a prime). The slightly modified version of their theorem will enable us to extend the results of \cite{BarJack1} and \cite{BarJack2} to sublines and subplanes for general $n$.

This paper is organised as follows.
 Field reduction, Desarguesian spreads, indicator sets and subspreads are introduced in Subsection \ref{fieldreduction}. We use explicit coordinates for the ABB-representation which are introduced in Subsection \ref{ABB}. This enables us to determine the indicator sets for the Desarguesian spread considered in this paper in an explicit way in Subsection \ref{indicator}. To allow us to use coordinates in the most convenient form in the calculations of Sections \ref{sublines} and \ref{subplanes}, we determine in Subsection \ref{InducedGroupAction} the induced action of the stabiliser of the line at infinity of $\PG(2,q^n)$ on the points of the ABB-representation in $\PG(2n,q)$.

In Section \ref{sublines}, we characterise $\F_{q^k}$-sublines meeting the line at infinity (Theorem \ref{tangentsubline}) and $\F_q$-sublines disjoint from the line at infinity (Theorem \ref{externalsubline}) and $\F_{q^k}$-sublines contained in the line at infinity (Theorem \ref{secantsubline}).
In Section \ref{subplanes}, we characterise $\F_{q^k}$-subplanes that are secant to the line at infinity (Theorem \ref{secantsubplane}) and $\F_q$-subplanes that are tangent to the line at infinity (Theorem \ref{tangentsubplane}). 

We begin by writing explicit representations for the various projective spaces that will be used in this article.

Consider a vector space $V \simeq \FF_{q^{n_{0}}} \times \cdots \times \FF_{q^{n_s}}$, for some positive integers $n_i$. A point $P$ of the projective space $\PG(V,\Fq)$ defined by the vector $v=(a_0,\ldots,a_s)$, where $a_i \in \FF_{q^{n_i}}$, will be written as $(v)_{\Fq}$ or $(a_0,\ldots,a_s)_{\Fq}$, emphasizing the fact that every $\F_q$-multiple of $v=(a_0,\ldots,a_s)$ gives rise to the point $P$, i.e.
\[(v)_{\Fq}=\{\lambda v\mid  \lambda \in \F_q^*\}\ \mathrm{ and }\
(a_0,\ldots,a_s)_{\Fq} = \{(\lambda a_0,\ldots,\lambda a_s) \mid \lambda\in \Fq^*\}.
\]

A {\em frame} of 
$\PG(n,q)$ is a set of $n+2$ points such that any $n+1$ points span $\PG(n,q)$.



Consider a subfield $\F_{q_0^{}}$ of $\F_q$; an {\em $n$-dimensional subgeometry} of $\PG(n,q)$ of order $q_{0}^{}$ is (the inherited incidence structure of) a set of $(q_0^{n+1}-1)/(q_0^{}-1)$ points whose homogeneous coordinates, with respect to a well-chosen frame of $\PG(n,q)$, are in $\F_{q_0^{}}$.

Using the previous definition, every $k$-dimensional subgeometry of a $k$-space contained in $\PG(n,q)$ is just called a $k$-dimensional subgeometry of $\PG(n,q)$.


\subsection{Field reduction, Desarguesian spreads, indicator sets and subspreads}\label{fieldreduction}
A {\em $(t-1)$-spread} $\S$ of $\PG(r-1,q)$ is a partition of the point set of $\PG(r-1,q)$ into subspaces of dimension $(t-1)$. Clearly, a $(t-1)$-spread of $\PG(r-1,q)$ can only exist if $t$ divides $r$. The construction of a {\em Desarguesian spread} that follows shows the well-known fact that this condition is also sufficient.

A Desarguesian spread of $\PG(rn-1,q)$ can be obtained by applying {\em field reduction} to the points of $\PG(r-1,q^n)$. The underlying vector space of the projective space $\PG(r-1,q^n)$ is $V(r,q^n)$; if we consider $V(r,q^n)$ as a vector space over $\F_q$, then it has dimension $rn$, so it defines a $\PG(rn-1,q)$. In this way, every point $P$ of $\PG(r-1,q^n)$ corresponds to a subspace of $\PG(rn-1,q)$ of dimension $(n-1)$ and it is not hard to see that this set of $(n-1)$-spaces forms a spread of $\PG(rn-1,q)$, which is called a Desarguesian spread. When we apply field reduction to an $\F_q$-subline of $\PG(r-1,q^n)$, we obtain a set of $q+1$ $(n-1)$-spaces forming a {\em regulus}, which is a set $\R$ of $(n-1)$-spaces contained in a $(2n-1)$-dimensional space, such that a line meeting $3$ elements of $\R$ meets all elements of $\R$. For more information about field reduction and Desarguesian spreads, we refer to \cite{FQ11}.

By Segre \cite{segre}, a Desarguesian spread can also be constructed as follows. Embed $\Lambda \simeq \PG(rn-1, q)$ as a subgeometry of $\Lambda^*\simeq \PG(rn-1, q^n)$. The subgroup of $\PGammaL(rn,q^n)$ fixing $\Lambda$ pointwise is isomorphic to $\Aut(\F_{q^n}/\Fq)$. Consider a generator $\sigma$ of this group. One can prove that that there exists an $(r - 1)$-space $\nu$ skew to the subgeometry $\Lambda$ and that a subspace of $\PG(rn-1, q^n)$ of dimension $s$ is fixed by $\sigma$  if and only if it intersects the subgeometry $\Lambda$ in a subspace of dimension $s$  (see \cite{IndicatorSet}). Let $P$ be a point of $\nu$ and let $L(P)$ denote the $(n-1)$-dimensional subspace generated by the {\em conjugates} of $P$, i.e., $L(P) = \langle P,P^\sigma,\ldots,P^{\sigma^{n-1}}\rangle$. Then $L(P)$ is fixed by $\sigma$ and hence it intersects $\PG(rn-1,q)$ in a $(n-1)$-dimensional subspace. Repeating this for every point of $\nu$, one obtains a set $\mathcal{D}$ of $(n-1)$-spaces of the subgeometry $\Gamma$ forming a spread. This spread $\D$ is a Desarguesian spread and $\{\nu,\nu^{\sigma},\ldots,\nu^{\sigma^{n-1}}\}$ is called the {\it indicator set} of $\D$. An indicator set is sometimes also called a set of {\em director spaces} \cite{segre} or a set of {\em transversal spaces} \cite{BarJack1}.
It is known from \cite{IndicatorSet} 
that for any Desarguesian $(n-1)$-spread of $\PG(rn-1,q)$ there exist a unique indicator set in $\PG(rn-1,q^n)$.


A $(k-1)$-{\em subspread} of an $(n-1)$-spread $\S$ of $\PG(rn-1,q)$, $k|n$, is a $(k-1)$-spread of $\PG(rn-1,q)$ that induces a $(k-1)$-spread in each element of $\S$. We can construct a Desarguesian $(k-1)$-subspread of the Desarguesian spread $\D$ as follows.

For every divisor $k|n$, we can consider the $(rn-1)$-dimensional subgeometry $\Lambda_{k} := \mathrm{Fix}(\sigma^k) \simeq \PG(rn-1,q^k)$ of $\Lambda^*$. Obviously, $\Lambda$ is contained in $\Lambda_k$.
Consider the $(\frac{rn}{k}-1)$-dimensional subgeometry $\Pi=\langle \nu, \nu^{\sigma^k},\ldots,\nu^{\sigma^{n-k}}\rangle \cap \Lambda_k$, this space is disjoint from $\Lambda$. One can see that the set $\{ \Pi, \Pi^\sigma, \ldots, \Pi^{\sigma^{k-1}}\}$ is the indicator set of a $(k-1)$-spread $\D_k$ in $\Lambda$. Clearly $\D_k$ is Desarguesian.

Consider a spread element $E \in \D$ and its $\Fqn$-extension $E^*$ in $\Lambda^*$. There exists a unique point $P \in \nu$ such that $E^*=\langle P,P^{\sigma},\ldots,P^{\sigma^{n-1}}\rangle$.
Consider the $(\frac{n}{k}-1)$-dimensional subgeometry $\pi=\langle P, P^{\sigma^k},\ldots,P^{\sigma^{n-k}}\rangle \cap \Lambda_k$ in $E^*$; this is a subspace of $\Pi$. The set $\{ \pi, \pi^\sigma, \ldots, \pi^{\sigma^{k-1}}\}$ is the indicator set of a $(k-1)$-spread of $E$ and each of these $(k-1)$-spaces is a spread element of $\D_k$. Hence the spread $\D_k$ induces a $(k-1)$-spread in each $(n-1)$-space of $\D$. It follows that $\D_k$ is a Desarguesian subspread of $\D$.

In \cite[Theorem 2.4]{squinn} the authors proved that there is a unique Desarguesian 1-subspread of a Desarguesian 3-spread in $\PG(7,q)$.
This is true in general, we will prove the following corollary in Section \ref{Sectionsecantsublines}.
\begin{corollary1}
A Desarguesian $(n-1)$-spread of $\PG(rn-1,q)$ has a unique Desarguesian $(k-1)$-subspread for each $k|n$.
\end{corollary1}

\begin{remark}
As described in the beginning of this section, a Desarguesian $(n-1)$-spread of $\PG(rn-1,q)$ can be obtained by applying field reduction to $\PG(r-1,q^n)$. Now consider a field reduction map $\mathcal{F}$ from subspaces of $\PG(r-1,q^n)$ to subspaces of $\PG(rn-1,q)$:
\[ \mathcal{F}: \PG(r-1,q^n)\rightarrow \PG(rn-1,q).\]
For a divisor $k|n$, this map can be written as the composition of two other field reduction maps $\mathcal{F}=\mathcal{F}_2 \circ \mathcal{F}_1$:
\[ \PG(r-1,q^n) \xrightarrow[\mathcal{F}]{} \PG(rn-1,q)\] \[= \PG(r-1,q^n) \xrightarrow[\mathcal{F}_1]{} \PG(\frac{rn}{k}-1,q^k) \xrightarrow[\mathcal{F}_2]{}\PG(rn-1,q). \]

If $\D$ is the Desarguesian $(n-1)$-spread in $\PG(rn-1,q)$ obtained by applying the field reduction map $\mathcal{F}$ to the points of $\PG(r-1,q^n)$, then its subspread $\D_k$ is the Desarguesian $(k-1)$-spread in $\PG(rn-1,q)$ obtained by applying the field reduction map $\mathcal{F}_2$ to the points of $\PG(\frac{rn}{k}-1,q^k)$.
\end{remark}

\subsection{The Andr\'e/Bruck-Bose representation of $\PG(2,q^n)$ in $\PG(2n,q)$}\label{ABB}

Andr\'e \cite{Andre} and Bruck and Bose \cite{Bruck-Bose} independently found a representation of
translation planes of order $q^n$ with kernel containing $\F_q$ in the projective
space $\PG(2n, q)$. We refer to this as the {\it Andr\'e/Bruck-Bose representation} or the {\it ABB-representation}. The construction of Andr\'e was based on group theory, Bruck and Bose
gave an equivalent geometric construction, which is the form we use in this paper and goes as follows.

Let $\S$ be a $(n- 1)$-spread in $\PG(2n - 1, q)$. Embed $\PG(2n - 1, q)$ as a hyperplane $H_\infty$ in $\PG(2n, q)$. Consider the following incidence structure $A(\S) = (\P,\L)$, where incidence is natural:
\begin{itemize}
\item[$\P:$] the affine points, i.e. the points of $\PG(2n, q) \backslash H_\infty$,
\item[$\L:$] the $n$-spaces of $\PG(2n, q)$ intersecting $H_\infty$ exactly in an element of $\S$.
\end{itemize}
In \cite{Bruck-Bose} the authors showed that $A(\S)$ is an affine translation plane of order $q^n$, and conversely, every such translation plane can be constructed in this way. If the spread $\S$ is Desarguesian, the plane $A(\S)$ is a Desarguesian affine plane $\AG(2, q^n)$.
The {\em projective completion} of the affine plane $A(\S)$ can be found by adding $H_\infty$ as the line $l_\infty$ at infinity where the elements of $\S$ correspond to the points of $l_\infty$, and we denote this by $\overline{A(\S)}$. Clearly, the projective completion $\overline{A(\S)}$ is a Desarguesian projective plane $\PG(2, q^n)$ if and only if the spread $\S$ is Desarguesian.

For any affine subspace $\pi$ of an affine geometry $\AG(N,q)=\PG(N,q)\setminus H_\infty$, we will denote the projective completion of $\pi$ by $\overline{\pi}$.

We now consider a specific way of giving coordinates to $\PG(2,q^n)$ and $\PG(2n,q)$, and using this we define an explicit map $\phi: \PG(2,q^n) \rightarrow \PG(2n,q)$ mapping each point to its corresponding element in the ABB-representation.

Recall that a point $P$ of $\PG(2,q^n)$ defined by a vector $(a,b,c) \in (\Fqn)^3$ is denoted by $(a,b,c)_{\Fqn}$. We fix a line at infinity of $\PG(2,q^n)$, say $l_\infty$, such that \[l_\infty= \{(a,b,0)_{\Fqn} \mid a,b \in \Fqn, (a,b) \ne (0,0)\}.\]
The affine points are the points of $\PG(2,q^n)\backslash l_{\infty}$ and clearly every affine point can be written as $(a,b,1)_{\Fqn}$, $a,b \in \Fqn$.

On the other hand, each point of $\PG(2n,q)$ can be denoted by $(a,b,c)_{\Fq}$, $a,b \in \Fqn, c \in \Fq$.
We consider the hyperplane $H_\infty$ with the following coordinates
\[H_{\infty} = \{(a,b,0)_{\Fq} \mid a,b \in \Fqn, (a,b)\neq(0,0) \}.\]
Furthermore, $H_{\infty}$ contains the Desarguesian $(n-1)$-spread $\D$ defined by
\[
\D = \left\{ \{(ax,bx,0)_{\Fq} \mid x \in \F_{q^n}^* \}\mid a,b \in \Fqn, (a,b)\neq(0,0)  \right\}.
\]

It is clear that the following map $\phi$, for $a,b \in \Fqn, (a,b)\neq(0,0)$, corresponds to the ABB-representation.
\begin{align*}
\phi:  \PG(2,q^n) &\rightarrow \PG(2n,q)\\
 (a,b,0)_{\Fqn} &\mapsto \{(ax,bx,0)_{\Fq}\mid x \in \Fqn^*\}\\
  (a,b,1)_{\Fqn} &\mapsto (a,b,1)_{\Fq}.
\end{align*}
The map $\phi$ will also be called the {\em ABB-map}.

\subsection{Choosing the right coordinates}\label{indicator}

Now consider the Desarguesian $(n-1)$-spread $\D$ in $H_\infty\simeq\PG(2n-1,q)$ constructed in the previous section and the embedding of $H_\infty$ as a hyperplane of $\Sigma\simeq\PG(2n,q)$.

We wish to consider an embedding of $\Sigma$ as a subgeometry in $\Sigma^*\simeq\PG(2n,q^n)$, such that the induced embedding of $H_\infty \subset \Sigma$ in $H^*_\infty \subset \Sigma^*$ provides us with a convenient description of the indicator set of $\D$ consisting of lines.

We can denote points of $\Sigma^* \simeq \PG(2n,q^n)$ by
\[
(a_0,a_1,\ldots,a_{n-1};b_0,b_1,\ldots,b_{n-1};c)_{\Fqn}, \mbox{ for } a_i,b_i,c \in \Fqn.
\]

Define the hyperplane $H_{\infty}^* \simeq \PG(2n-1,q^n)$ to consist of the points of $\Sigma^*$ for which $c=0$.

We define a collineation $\sigma$ of $\Sigma^*$ by
\begin{align*}
\sigma: \PG(2n,q^n) \rightarrow \PG(2n,q^n) &:(a_0,a_1,\ldots,a_{n-1};b_0,b_1,\ldots,b_{n-1};c)_{\Fqn} \\
&\mapsto (a_{n-1}^q,a_0^q,\ldots,a_{n-2}^q;b_{n-1}^q,b_0^q,\ldots,b_{n-2}^q;c^q)_{\Fqn}.
\end{align*}
for $a_i,b_i \in \Fqn$, $c \in \Fq$. The corresponding map on the vector defining a point of $\PG(2n,q^n)$ will also be denoted by $\sigma$. The points of $\Sigma^*$ fixed by $\sigma$ form a subgeometry isomorphic to $\PG(2n,q)$; this subgeometry is the following: \[\{(a,a^q,\ldots,a^{q^{n-1}};b,b^q,\ldots,b^{q^{n-1}};c) _{\Fqn}\mid a,b \in \Fqn, c \in \Fq, (a,b,c)\neq (0,0,0)\}.\]
Hence we can see the embedding of $\Sigma \simeq \PG(2n,q)$ in $\Sigma^* \simeq \PG(2n,q^n)$ via the following map $\iota$, for $a,b \in \Fqn$, $c \in \Fq$.
\begin{align*}
\label{eqn:embedPG2n}
\iota: \PG(2n,q)&\rightarrow \PG(2n,q^n)\\
 (a,b,c)_{\Fq} &\mapsto (a,a^q,\ldots,a^{q^{n-1}};b,b^q,\ldots,b^{q^{n-1}};c) _{\Fqn}.
\end{align*}


Clearly $$\iota(H_{\infty})=\{(a,a^q,\ldots,a^{q^{n-1}};b,b^q,\ldots,b^{q^{n-1}};0) _{\Fqn} \mid a,b \in \Fqn, (a,b)\neq(0,0)\}$$ forms a $(2n-1)$-dimensional subgeometry of order $q$ of $H^*_{\infty}$.

Let us now consider the line $\nu$ in $H^*_{\infty}$, disjoint from $\iota(H_\infty)$, defined as
\[
\nu =  \left\langle(1,0,\ldots,0;0,0,\ldots,0;0),(0,0,\ldots,0;1,0,\ldots,0;0)\right\rangle_{\Fqn}.
\]
Then the set $\{\nu,\nu^{\sigma},\ldots,\nu^{\sigma^{n-1}}\}$ is an indicator set defining a Desarguesian spread of $\iota(H_{\infty})$ consisting of the $(n-1)$-spaces
\[
\left\{ (ax,(ax)^q,\ldots,(ax)^{q^{n-1}};bx,(bx)^q,\ldots,(bx)^{q^{n-1}};0)  _{\Fqn} \mid x \in \F_{q^n}^*\right\},
\]
for $a,b \in \Fqn$, $(a,b)\neq(0,0)$. It is easy to see that this is precisely $\iota(\D)$. By abuse of notation, from now on, we will denote $\iota(H_{\infty})$ and $\iota(\D)$ again by $H_{\infty}$ and $\D$.


%

One can check that the $(k-1)$-subspread $\D_k$ of $\D$ in its original setting of $H_\infty$ (see Subsection \ref{fieldreduction}) corresponds to the set
\[
\D_k = \left\{ \{(ax,bx,0)_{\Fq} \mid x \in \F_{q^k}^* \}\mid a,b \in \Fqn, (a,b)\neq(0,0)  \right\}.
\]
This implies that $\iota(\D_k)$, which we also denote by $\D_k$, corresponds to the set of $(k-1)$-spaces of the form
\[
\left\{ (ax,(ax)^q,\ldots,(ax)^{q^{n-1}};bx,(bx)^q,\ldots,(bx)^{q^{n-1}};0)  _{\Fqn} \mid x \in \Fqk\right\},
\]
for $a,b \in \F_{q^n}$, $(a,b)\neq(0,0)$.


\subsection{The induced action of the stabiliser of $l_{\infty}$ on $\PG(2n,q^n)$}\label{InducedGroupAction}

We use the notations introduced in the previous sections. Recall that points $\PG(2,q^n)$ have coordinates of the form $(a,b,c)_{\Fqn}$, where $a,b,c \in \Fqn$ and that the line $l_\infty$ meets equation $c=0$. In this section we consider how an element of the stabiliser of $l_{\infty}$ in $\PGL(3,q^n)$, say $G$, induces an action on $\Sigma \simeq \PG(2n,q)$ and we describe its extension to an element of $\PGammaL(2n+1,q^n)$ acting on $\Sigma^* \simeq \PG(2n,q^n)$. This will be of use later in this paper since it will allow us to study the representation of a representative of orbits of sublines or subplanes under $G$.

Every element $\chi_{{}_{0}}$ of the stabiliser $G$ of $l_{\infty}$ corresponds to a matrix of the form
\[
X=\npmatrix{ x_{11}&x_{12}&0\\x_{21}&x_{22}&0\\x_{31}&x_{32}&1},
\]
where $x_{ij} \in \Fqn$ and $x_{11}x_{22}-x_{12}x_{21} \ne 0$, where we let the matrix act on row vectors from the right, hence for $a,b,c \in \Fqn$:
\begin{align*}\chi_{{}_0}: \PG(2,q^n) &\rightarrow \PG(2,q^n) \\
 (a,b,c)_{\Fqn} &\mapsto \left( (a,b,c)X \right)_{\Fqn}.\end{align*}
The map $\chi_{{}_0}$ induces a natural action $\chi$ on the points $(a,b,c)_{\Fq}$ of $\Sigma$, $a,b \in \Fqn$, $c\in\Fq$, as follows:
\begin{align*}\chi: \PG(2n,q) &\rightarrow \PG(2n,q) \\
 (a,b,c)_{\Fq}&\mapsto ((a,b,c)X)_{\Fq}.\end{align*}

Recall that we denoted points of $\Sigma^* \simeq \PG(2n,q^n)$ by
\[
(a_0,a_1,\ldots,a_{n-1};b_0,b_1,\ldots,b_{n-1};c)_{\Fqn}, \mbox{ for } a_i,b_i,c \in \Fqn.
\]
For shorthand, we will write these now as $\left((a_i);(b_i);c\right)_{\Fqn}$, where the index $i$ is assumed to range from $0$ to $n-1$.

We define the \emph{extension} of $\chi$, denoted by $\chi^*$, to be the collineation of $\Sigma^*$ which acts on a generic point as follows:
\begin{align*}
\chi^*:\PG(2n,q^n)&\rightarrow \PG(2n,q^n)\\
((a_i);(b_i);c)_{\F_{q^n}} &\mapsto (({x_{11}}^{q^i} a_i + {x_{21}}^{q^i} b_i + {x_{31}}^{q^i} c);( {x_{12}}^{q^i} a_i + {x_{22}}^{q^i} b_i + {x_{32}}^{q^i} c); c)_{\F_{q^n}}.
\end{align*}
The following lemma is easy to prove.
 \begin{lemma} \label{IGA}Using the notations from above, we have the following.
  \vspace*{-0.3cm}
 \begin{enumerate}
\item[{\rm(a)}]
The map $\chi$ has the following properties:
\begin{itemize} \setlength{\itemsep}{0pt}
\item
$\chi$ stabilises the Desarguesian spread $\D_k$ of $H_{\infty}$ for each divisor $k$ of $n$,
\item
$\chi$ stabilises the set of $k$-spaces of $\PG(2n,q)$ meeting $H_{\infty}$ in an element of $\D_k$.
\end{itemize}
Moreover, let $\phi: \PG(2,q^n)\rightarrow \PG(2n,q)$ be the ABB-map as defined in Section \ref{ABB}. Then for every point $P$ of $\PG(2,q^n)$, we have that $$\phi \chi_{{}_0}(P)=\chi\phi(P).$$

\item[{\rm(b)}] The extension $\chi^*$ of $\chi$ has the following properties:
\begin{itemize} \setlength{\itemsep}{0pt}
\item $\chi^*(P^{\sigma}) = (\chi^*(P))^{\sigma}$ for all $P \in \Sigma^*$, i.e. $\chi^*$ maps conjugate points to conjugate points,
\item
$\chi^*$ stabilises the indicator set of $\D_k$ for each divisor $k$ of $n$.
\end{itemize}
Moreover, let $\iota: \PG(2n,q)\rightarrow \PG(2n,q^n)$ be the embedding defined in Section \ref{indicator}. Then for every point $P$ of $\PG(2n,q)$, we have that $$\iota \chi(P)=\chi^*\iota(P).$$
\end{enumerate}

\end{lemma}
\begin{proof} The proof follows from straightforward but tedious calculations.

(a) Notice that the image of an element $\{(ax,bx,0)_{\Fq}\}$ of $\D_k$ under $\chi$ is given by $\{((ax_{11}+bx_{21})x,(ax_{12}+bx_{22})x,0)_{\F_{q}}|x\in \F_{q^k}\}\subset \D_k$ from which the two properties of $\chi$ follow.

For the second statement of (a), we consider a point $P\in \PG(2,q^n)$ with coordinates $(a,b,c)_{\F_{q^n}}$. We have that $\chi_{{}_0}(P)=((a,b,c)X)_{\Fqn}= (ax_{11}+bx_{21}+cx_{31},ax_{12}+bx_{22}+cx_{32},c)_{\F_{q^n}}$, so
$\phi \chi_{{}_0}(P)=(ax_{11}+bx_{21}+cx_{31}, ax_{12}+bx_{22}+cx_{32},c)_{\F_{q}}$ and we get $\chi \phi(P)=\chi (a,b,c)_{\F_q}=(ax_{11}+bx_{21}+cx_{31},ax_{12}+bx_{22}+cx_{32},c)_{\F_{q}}.$

(b) Consider a point $P$ with coordinates $((a_i);(b_i);c)_{\F_{q^n}}$ in $\Sigma^*$. If we put $a_{-1}=a_{n-1}$ and $b_{-1}=b_{n-1}$, then for $i\in \{0,\ldots n-1\}$, the $(i+1)$-th coordinate of $\chi^*(P)$ is $x_{11}^{q^{i}}a_{i}+x_{21}^{q^{i}}b_{i}+x_{31}^{q^{i}}c$, which implies that the $(i+1)$-th coordinate of $(\chi^*P)^\sigma$ is $x_{11}^{q^{i+1}}a_{i-1}^q+x_{21}^{q^{i+1}}b_{i-1}^q+x_{31}^{q^{i+1}}c$. This is clearly equal to the $(i+1)$-th coordinate of $\chi^*(P^\sigma)$. The same argument holds for the other coordinate positions.
This also implies that $\chi^*$ stabilises the indicator sets.

For the last statement, consider a point $P$ in $\PG(2n,q)$ with coordinates $(a,b,c)_{\F_{q^n}}$, then $\iota(P)=((a^{q^i});(b^{q^i});c)_{\F_{q^n}}$ and
\begin{align*}
\chi^*\iota(P)&=\left((x_{11}^{q^i}a^{q^i}+x_{21}^{q^i}b^{q^i}+x_{31}^{q^i}c);(x_{12}^{q^i}a^{q^i}+x_{22}^{q^i}b^{q^i}+x_{32}^{q^i}c);c\right)_{\Fqn} \\
&=\left(((x_{11}a+x_{21}b+x_{31}c)^{q^i});((x_{12}a+x_{22}b+x_{32}c)^{q^i});c\right)_{\Fqn}\\
&=\iota\chi(P).\qedhere\end{align*}\end{proof}

\section{The ABB-representation of sublines}\label{sublines}

In this section we will determine the ABB-representation of $\F_{q^k}$-sublines of $\PG(2,q^n)$.  We need to make a distinction between sublines that are tangent to, external to or contained in $l_{\infty}$.

We start with the first two cases, which will be handled by use of coordinates. We end with the case of $\F_{q^k}$-sublines contained in $l_{\infty}$, for which no coordinates are needed.

\subsection{The equivalence of sublines under the stabiliser of $l_\infty$}

In the case of tangent and external sublines, we will consider sublines with specific coordinates. As before, we consider the plane $\PG(2,q^n) = \{ (a,b,c)_{\Fqn} \mid a,b,c \in \Fqn\}$ and the line $l_\infty$ with equation $c=0$. Let $l$ be the line meeting the equation $a=0$. We will show in Lemma \ref{equivalentsublines} that any subline tangent or external to $l_\infty$ is equivalent to a particular subline contained in $l$. Note that an $\F_{q^k}$-subline of $\PG(2,q^n)$, $k|n$, is uniquely determined by three collinear projective points, or, equivalently, by two distinct vectors of $\F_{q^n}^3$.

Given $\omega \in \Fqn$ and a divisor $k|n$, denote by $l_{\omega,k}$ the $\F_{q^k}$-subline determined by the vectors $(0,1,\omega),(0,0,1)$, or alternatively by the points $(0,1,\omega)_{\Fqn}$, $(0,0,1)_{\Fqn}$, $(0,1,\omega+1)_{\Fqn}$, i.e.
\[
l_{\omega,k} = \left\{(0,1,\omega+ t)_{\Fqn} \mid t \in \F_{q^k} \}\cup \{(0,0,1)_{\Fqn}\right\}.
\]

Consider a set of sublines; we refer to a subline being the {\em smallest} of the set, if it has the smallest number of points or equivalently if it is the subline over the smallest field.

\begin{lemma}\label{smallest}
If $\Fqk = \Fq(\omega)$ is the smallest subfield of $\Fqn$ containing $\omega$, then the subline $l_{\omega,k}$ is the smallest subline tangent to $l_{\infty}$ and containing $l_{\omega,1}$.
\end{lemma}
\begin{proof}
For all $k|n$, there is a unique $\F_{q^k}$-subline containing the points $(0,1,\omega)_{\Fqn}$, $(0,0,1)_{\Fqn}$, $(0,1,\omega+1)_{\Fqn}$. Hence, the $\F_{q^k}$-subline containing the points of $l_{\omega,1}$ is of the form $l_{\omega,k}$. The subline $l_{\omega,k}$ contains the point $(0,1,0)_{\Fqn}$ of $l_\infty$ if and only if $-\omega\in \F_{q^k}$, and the statement follows.
\end{proof}

\begin{lemma}\label{equivalentsublines}
Every external $\F_{q^k}$-subline to $l_{\infty}$ is equivalent under the action of the stabiliser of $l_{\infty}$ in $\PGammaL(3,q^n)$ to $l_{\omega,k}$ for some $\omega\notin \Fqk$.

Every $\F_{q^k}$-subline tangent to $l_{\infty}$ is equivalent under the action of the stabiliser of $l_{\infty}$ in $\PGammaL(3,q^n)$ to $l_{0,k}$.
\end{lemma}
\begin{proof}
Consider an $\F_{q^k}$-subline $m$ not contained in $l_\infty$. Then $m$ is determined by two vectors in $\Fqn^3$, which we may take to be $(\alpha,\beta,1)$ and $(\gamma,\delta,\omega)$. Then the matrix
\[
\label{eqn:stabilizermatrix}
\npmatrix{u&v&0\\ \gamma -\omega \alpha & \delta - \omega \beta&0\\\alpha&\beta&1},
\]
where $u,v$ are chosen so that this matrix is invertible, maps $(0,0,1)$ to $(\alpha,\beta,1)$ and $(0,1,\omega)$ to $(\gamma, \delta,\omega)$, and hence defines a collineation $\psi$ in the stabiliser of $l_\infty$ which maps $l_{\omega,k}$ to $m$.

As $\psi$ stabilises $\l_{\infty}$, if $m$ is tangent to $l_{\infty}$, then $l_{\omega,k}$ is as well. The subline $l_{\omega,k}$ meets $l_{\infty}$ if and only if $\omega \in \Fqk$ by Lemma \ref{smallest}, in which case $l_{\omega,k} = l_{0,k}$, proving the claim.
\end{proof}

%
%
%
%
%
%
%
%

\subsection{Sublines tangent to $l_{\infty}$}\label{Sectiontangentsublines}
\begin{theorem}\label{tangentsubline}
\hspace*{0.1cm}
\begin{itemize}
\item[{\rm (a)}]
The affine points of an $\F_{q^k}$-subline $m$ in $\PG(2,q^n)$ tangent to $l_\infty$ correspond to the points of a $k$-dimensional affine space $\pi$ in the ABB-representation, such that $\overline{\pi} \cap H_\infty$ is an element of $\D_k$.

\item[{\rm (b)}]
Conversely, let $\pi$ be a $k$-dimensional affine space of $\Sigma$ such that $\overline{\pi}$ intersects $H_\infty$ in a spread element of $\D_k$. Then the points of $\pi$ correspond to the affine points of an $\F_{q^k}$-subline $m$ tangent to $l_\infty$.
\end{itemize}
\end{theorem}
\begin{proof}
(a)
From Lemma \ref{equivalentsublines}, the tangent $\F_{q^k}$-subline $m$ is equivalent to $l_{0,k}$ under an element of the stabiliser $G$ of $l_\infty$ in $\PGL(3,q^n)$, say $\chi_{{}_0}(l_{0,k})=m$ for $\chi_{{}_0}\in G$. Note that $l_{0,k}$ consists of the following points:
\[
l_{0,k} 
     = \left\{ (0,b,1)_{\Fqn} \mid b \in \F_{q^k} \right\} \cup \left\{(0,1,0)_{\Fqn}\right\}.
\]
In the ABB-representation the point $(0,1,0)_{\Fqn} \in l_{0,k}$ corresponds to the spread element $\{ (0,x,0)_{\Fq} \mid x \in \F_{q^n}^*\} \in \D$.
By definition of the ABB-map $\phi$, the affine points of $l_{0,k}$ in the ABB-representation form a set $\pi$ defined as follows:
\[
\pi = \left\{ \left(0,b,1\right)_{\Fq} \mid b \in \F_{q^k} \right\}.
\]
Using the embedding $\iota$ of $\pi$ in $\PG(2n,q^n)$ we obtain the set of points
\[
 \left\{ (0,\ldots,0;b,b^q,\ldots,b^{q^{n-1}};1)_{\Fqn} \mid b\in \F_{q^k} \right\}.
\]
Since $b^{q^{k+j}}=b^{q^j}$ for all $b \in \Fqk$, it is clear that the projective completion $\overline{\pi}$ of $\pi$ intersects $H_\infty$ in an element of $\D_k$, more specifically in the element $\langle Q, Q^\sigma, \ldots, Q^{\sigma^{k-1}}\rangle \cap H_{\infty}$, where
\begin{align*}
Q&=(v+v^{\sigma^k}+v^{\sigma^{2k}}+\ldots+v^{\sigma^{n-k}})_{\F_{q^n}},\\
v &= (0,0,\ldots,0;1,0,\ldots,0;0),\ (v)_{\F_{q^n}}\in \nu.
\end{align*}

We know that $\phi(m)=\phi(\chi_{{}_0}(l_{0,k}))$, and that the affine points of $\phi(l_{0,k})$ form the point set of an affine space $\pi$ such that its projective completion $\overline{\pi}$ intersects $H_\infty$ in an element of $\D_k$. By Lemma \ref{IGA}, we know that $\phi\chi_{{}_0}=\chi\phi$ and that $\chi$ stabilises the set of $k$-spaces meeting $H_\infty$ in an element of $\D_k$. This implies that also the affine points of $\phi(m)$ form the point set of an affine space such that its projective completion intersects $H_\infty$ in an element of $\D_k$.

(b) To prove that the converse also holds, it is sufficient to use a counting argument.

The number of affine points of $\PG(2,q^n)\setminus l_{\infty}$ is equal to the number affine points of $\PG(2n,q)\setminus H_{\infty}$. Hence the number of choices for any two distinct affine points is the same in both cases.

For any $k$, two affine points $A_1, A_2$ in $\PG(2,q^n)$ define a unique $\F_{q^k}$-subline tangent to $l_\infty$. Because of (a) this subline corresponds to a $k$-dimensional space intersecting $H_\infty$ in an element of $\D_k$.

The two affine points $A_1, A_2$ correspond in the AAB-representation to two affine points $B_1, B_2$ in $\PG(2n,q)$. There is a unique element of $\D_k$ containing the point $\langle B_1, B_2 \rangle \cap H_\infty$, hence there is a unique $k$-dimensional space containing $B_1$, $B_2$ and intersecting $H_\infty$ in an element of $\D_k$.

From this we see that the number of $\F_{q^k}$-sublines tangent to $l_\infty$ is equal to the number of $k$-spaces intersecting $H_\infty$ in an element of $\D_k$. Hence, the statement follows.
\end{proof}
\subsection{Sublines disjoint from $l_{\infty}$}

A {\em $u$-arc} $\A$ in $\PG(k,q)$ is a set of $u$ points in $\PG(k,q)$, such that every subset of $k+1$ points of $\A$ span $\PG(k,q)$. A $u$-arc in $\PG(k,q)$ is also called a set of $u$ points {\em in general position}.
\begin{definition}
A {\em normal rational curve} in $\PG(k,q)$, $2 \leq k \leq q$, is
a $(q+1)$-arc $\PGL$-equivalent to the $(q+1)$-arc
\[\{(0,\ldots,0,1)_{\Fq}\}\cup\{  (1,t,t^2,t^3,\ldots,t^{k})_{\F_q} \mid t\in \F_q\}.\]

A set $\C$ in $\PG(N,q)$ is a normal rational curve of {\em degree $k$} if and only if it is a normal rational curve in a $k$-dimensional subspace of $\PG(N,q)$, or equivalently, if and only if there exist linearly independent vectors $e_0,e_1,\ldots, e_k$ in $V(N+1,q)$ such that
\begin{align*}
\C &= \{ ( s^k e_0+s^{k-1} t e_1 + \ldots + s t^{k-1}e_{k-1} + t^k e_k  )_{\Fq} \mid s,t \in \Fq \}.
\end{align*}
\end{definition}
A normal rational curve of degree 1 is just a projective line.

\begin{result}\label{NRC}{\rm \cite[Theorem 21.1.1]{FPS3D}}
Consider a $(k + 3)-arc$ $\A$ in $\PG(k,q)$, $ k+2 \leq q$, then there exists a
unique normal rational curve containing all points of $\A$.
\end{result}

To illustrate this, consider $k+3$ points defined by vectors $u_0,\ldots,u_k,a,b$, where $a = \sum_{i=0}^k a_i u_i$ and  $b = \sum_{i=0}^k b_i u_i$. Then the unique normal rational curve through these points may be parametrised as
\[
\left\{ \left(\sum_{i=0}^k \prod_{j=0,j \ne i}^k (a_j s - b_j t) u_i\right)_{\Fq} \middle| \ s,t \in \Fq \right\}.
\]
Then the point $(u_j)_{\F_q}$ corresponds to taking $(s,t)=(b_j,a_j)$, the point $(a)_{\F_q}$ to $(1,0)$ and the point $(b)_{\F_q}$ to $(0,1)$.

Consider a normal rational curve $\C$ of $\PG(k,q)$, $2\leq k \leq q$, and the embedding of $\PG(k,q)$ as a subgeometry of $\PG(k,q^n)$. Because of the previous result, a normal rational curve $\C^*$ in $\PG(k,q^n)$ containing the points of $\C$ is unique and we call this the \emph{$\F_{q^n}$-extension} $\C^*$ of $\C$.

As we have seen, a normal rational curve $\C$ is the point set of an algebraic variety in $\PG(k,q)$ defined by the parameter $t \in \Fq$. The extension $\C^*$ of $\C$ can be obtained as the point set from the variety which is obtained by allowing the parameter $t$ to range over $\Fqn$.

Before we get to the theorem characterising disjoint sublines, we need the following lemma.

\begin{lemma}\label{3pointsNRC}
Let $A,B,C$ be three non-collinear affine points of $\PG(2n,q)$, contained in an $n$-space containing an element $E$ of $\D$ and consider the line $l := \langle A,B,C\rangle \cap H_{\infty}$. Consider the set $S$ of values $i$ for which some element $E_i$ of $\D_i$ contains the line $l$. 
For every $m\neq \min(S)$ in $S$, the points $A,B,C$ and the $m$ conjugate points generating the element $E_m$ are not in general position.
\end{lemma}

\begin{proof}
Consider the line $l := \langle A,B,C\rangle \cap H_{\infty}$. Let $k$ be the minimum of $S$. It then follows from the definition of $\D_i$ that $S$ is the set of all values $tk$ where $tk|n$.


Let $m=tk$ with $t>1$. Suppose $L(Q)=\{Q,Q^\sigma,\ldots,Q^{\sigma^{m-1}}\}$ is a set of conjugate points generating $E_m$. The $(k-1)$-space $E_k$ is part of the $(k-1)$-spread $\S=\D_k \cap E_m$. Consider the $(t-1)$-space $\Pi=\langle Q, Q^{\sigma^{k}}, \ldots, Q^{\sigma^{(t-1)k}}\rangle$. The set $\{\Pi, \Pi^\sigma, \ldots, \Pi^{\sigma^{k-1}}\}$ is the indicator set of the spread $\S$ of $E_m$. This implies that we can label the indicator set $L(P)=\{P,P^\sigma,\ldots,P^{\sigma^{k-1}}\}$ for $E_k$  in such a way that the point $P^{\sigma^i}$ is contained in $\Pi^{\sigma^{i}}$ for all $i$.

We prove that the $m+3$ points of the set $L(Q)\cup\{A,B,C\}$, which is contained in an $m$-dimensional space, are not in general position by constructing $u+2$ points that are contained in a $u$-space, where $u\leq m-1$.

The space $\langle L(P), A, B, C\rangle$ has dimension $k$, and every space $\Pi^{\sigma^{i}}$ intersects it in exactly one point, namely the point $P^{\sigma^i}$.
Consider the $(k-1)t+3$ points contained in $L(Q)\cup\{A,B,C\}$, but not contained in $\Pi^{\sigma^{k-1}}$. These points are contained in the space $\langle L(P), A, B,C, \Pi,\Pi^{\sigma},\ldots,\Pi^{\sigma^{k-2}}\rangle$ which has dimension at most $k+(k-1)(t-1)=(k-1)t+1$. Since $t>1$, we have that $u=(k-1)t+1\leq m-1=kt-1$ and our claim follows.
\end{proof}

The {\em norm map} of $\F_{q^k}$ over $\Fq$ is denoted as follows: $N_k(\alpha) = \prod_{i=0}^{k-1} \alpha^{q^i}\in \Fq$.

\begin{theorem}\label{externalsubline}
A set of points $\C$ in $\PG(2n,q)$, $q\geq n$, is the ABB-representation of an $\Fq$-subline $m$ of $\PG(2,q^n)$ external to $l_\infty$ if and only if
\begin{itemize}
\item[{\rm (i)}] $\C$ is a normal rational curve of degree $k$ contained in a $k$-space intersecting $H_{\infty}$ in an element of $\D_k$,
\item[{\rm (ii)}] its extension $\C^*$ to $\PG(2n,q^n)$ intersects the indicator set \\
$\{\Pi, \Pi^{\sigma}, \ldots, \Pi^{\sigma^{k-1}}\}$ of $\D_k$ in $k$ conjugate points.
    \end{itemize}
Moreover, the smallest subline containing $m$ and tangent to $l_\infty$ is an $\F_{q^k}$-subline.

%
%
%
\end{theorem}

\begin{proof}
Part (i) has been proven in \cite[Theorem 4.2]{LaZa201?}. We include a proof using coordinates, as it will be necessary for proving part (ii).


(i)
Consider the smallest $k$, such that $m$ is contained in a tangent $\Fqk$-subline. It follows from Theorem \ref{tangentsubline} that the $q+1$ points corresponding to $m$ are contained in a $k$-space intersecting $H_\infty$ in an element of $\D_k$.

By Lemmas \ref{smallest} and \ref{equivalentsublines} we know that $m$ is $G$-equivalent to the $\F_q$-subline $l_{\omega,1}$, say $m = \chi_{{}_0}(l_{\omega,1})$, where $\omega \in \Fqn$ such that $\Fq(\omega)=\Fqk$ and $\chi_{{}_0}$ in the stabiliser $G$ of $l_{\infty}$ in $\PGL(3,q^n)$.

By definition, $\phi(l_{\omega,1})$, where $\phi$ is the ABB-map, is a set $\C_\omega$ defined as follows:
\[
\C_{\omega} = \left\{ \left(0,\frac{1}{\omega+ t},1\right)_{\Fq} \middle| \ t \in \Fq \right\}\cup \left\{(0,0,1)_{\Fq}\right\}.
\]
Now $N_k(t+\omega) = \prod_{i=0}^{k-1} (\omega^{q^i}+t)\in \Fq$ for all $t \in \Fq$, and is never zero, and so
\[
\C_{\omega} = \left\{ \left(0,\prod_{i=1}^{k-1} (\omega^{q^i}+t),\prod_{i=0}^{k-1} (\omega^{q^i}+t)\right)_{\Fq} \middle| \ t \in \Fq \right\}\cup \left\{(0,0,1)_{\Fq}\right\}.
\]
Expanding the products, we find $k+1$  non-zero vectors $v_i \in \Fqn \times \Fqn \times \Fq$ (which depend only on $\omega$) such that
\[
\C_{\omega} = \{ ( v_0+t v_1 + \ldots + t^k v_k  )_{\Fq} \mid t \in \Fq \}\cup \{( v_k )_{\Fq}\}.
\]
The vectors $v_i$ span $0 \times \Fqk \times \Fq$, hence $\C_{\omega}$ is a normal rational curve of degree $k$, contained in a projective $k$-space meeting $H_{\infty}$ in an element of $\D_k$, as claimed.

(ii)
Using the embedding $\iota$ to embed $\C_{\omega}$ in $\PG(2n,q^n)$ gives the points
\[
 \left(0,0,\ldots,0;\prod_{i=1}^{k-1} (\omega^{q^i}+t),\prod_{i=1}^{k-1} (\omega^{q^{i+1}}+t),\ldots,\prod_{i=1}^{k-1} (\omega^{q^{i+n-1}}+t);\prod_{i=0}^{k-1} (\omega^{q^i}+t)\right)_{\Fqn},
\]
for $t \in \Fq$. Since $\omega \in \Fqk$, the $(n+j)$-th entry is equal to the $(n+j+k)$-th entry for all $0\leq j \leq n-1$. The extension $\C^*_{\omega}$ of $\C_{\omega}$ is the normal rational curve obtained by allowing $t$ to range over $\Fqn$. The intersection of $\C^*_{\omega}$ with ${H}_{\infty}^*$ is then obtained by finding the elements $t \in \Fqn$ such that $\prod_{i=0}^{k-1} (\omega^{q^i}+t)= 0$. Clearly these are precisely $t = -\omega, -\omega^q,\ldots,-\omega^{q^{k-1}}$. When we consider the value $t = -\omega^{q^j}$,  then all coordinates become zero, excluding those in position $(n+j+rk)$ for $r =0,1,\ldots,n/k-1$. Hence we get that the points of $\C^*_{\omega} \cap H^*_{\infty}$ are precisely the $k$ conjugate points as follows
\[
Q^{\sigma^j}=(v^{\sigma^j}+v^{\sigma^{j+k}}+v^{\sigma^{j+2k}}+\ldots +v^{\sigma^{j+(n-k)}})_{\F_{q^n}}
\]
where $j = 0,1,\ldots, k-1$ and $P = (v)_{\F_{q^n}}$ with $v=(0,0,\ldots,0;1,0,\ldots,0;0)$, so $P\in \nu$. If $k=n$, i.e. if $\omega$ is not contained in any proper subfield of $\Fqn$, then $\C^*_{\omega}$ contains $n$ conjugate points, one on each of the lines of the indicator set $\{\nu,\nu^\sigma,\ldots,\nu^{\sigma^{n-1}}\}$. However if $k<n$, each point belongs to one of the spaces of the set $\{\pi, \pi^\sigma,\ldots,\pi^{\sigma^{k-1}}\}$, where $\pi=\langle P, P^{\sigma^k}, \ldots, P^{\sigma^{n-k}}\rangle$. It is clear that $\pi \subset \Pi$, where $\{\Pi, \Pi^\sigma,\ldots,\Pi^{\sigma^{k-1}}\}$ is the indicator set of $\D_k$.



Now $m=\chi_{{}_0}(l_{\omega,1})$ and $\phi(m)=\phi(\chi_{{}_0}(l_{\omega,1}))=\chi(\phi(l_{\omega,1}))$ by Lemma \ref{IGA}. Since $\chi$ is a collineation, this implies that $\phi(m)$ is also a normal rational curve of degree $k$, and since $\chi$ stabilises the elements of $D_k$, this normal rational curve lies in a $k$-space intersecting $H_\infty$ in an element of $\D_k$. Now, again by Lemma \ref{IGA}, we have that $\iota\phi(m)=\iota\chi\phi(l_{\omega,1})=\chi^*\iota \phi(l_{\omega,1})$. Since $\chi^*$ stabilises the indicator sets, the unique $\Fqn$-extension of $\iota \phi(m)$ is also a normal rational curve which intersects the indicator set in conjugate points. This proves the first part of the claim.


To prove that the converse also holds, it is sufficient to use a counting argument. As three points on a line of $\PG(2,q^n)$ uniquely determine a subline, the number of choices for three points defining an external $\Fq$-subline of a fixed line is equal to ${q^n \choose 3} - \frac{q^n(q^n-1)(q-2)}{3.2}=\frac{q^n(q^n-1)(q^n-q)}{6}$.

By the first part of the proof, we know that three non-collinear affine points contained in an $n$-space containing an element of $\D$ are contained in a normal rational curve $\C$ satisfying (i) and (ii) for some value $k$. By Lemma \ref{3pointsNRC}, we know that through three non-collinear affine points, there is at most one value of $k$ (namely, $\min(S)$) such that there is a normal rational curve of degree $k$ satisfying (i) and (ii), so the obtained normal rational curve $\C$ is unique.  Hence it suffices to note that the number of triples of non-collinear affine points in a fixed $n$-space is $\frac{q^n(q^n-1)(q^n-q)}{6}$, which equals the number of external sublines of a fixed line, completing the proof.
\end{proof}

\begin{remark}
In the statement of \cite{BarJack1}, characterising the external sublines, the authors state that $\C^*$ meets each transversal line $\nu^{\sigma^j}$. This is due to the fact that the authors choose one representative for a subline, equivalent to our choice of $l_{\omega,1}$, where $\omega$ is a primitive element of $\F_{q^n}$. With this choice, the unique subline tangent to $l_\infty$ and containing $l_{\omega,1}$ always is the full line and in that case, we have indeed seen in Theorem \ref{externalsubline} that $\C^*$ meets each transversal line $\nu^{\sigma^j}$. But as follows from Lemma \ref{smallest}, unless $n$ is prime, there are many choices for a subline $l_{\omega,1}$ in $\PG(2,q^n)$ contradicting this statement.
\end{remark}

\begin{corollary}\label{external-Fqk-subline}
An $\Fqk$-subline $m$ in $\PG(2,q^n)$, $q\geq n$, external to $l_\infty$, such that the smallest subline containing $m$ and tangent to $l_\infty$ is an $\F_{q^r}$-subline, corresponds in the ABB-representation to a set $M$ of $q^k+1$ affine points no three collinear, contained in an $r$-space intersecting $H_\infty$ in an element of $\D_r$.
Every three points of $M$ determine an affine normal rational curve whose points are all contained in $M$.
\end{corollary}

Consider an $\Fqk$-subline $m$ of $\PG(1,q^n)$. The unique $\Fq$-subline defined by any three points of $m$ is completely contained in $m$. Moreover, this property gives a characterisation of an $\Fqk$-subline, as seen in the following result.

\begin{result}{\rm \cite{PsHyOv}} \label{moeilijk} A set $S$ of at least $3$ points in $\PG(1,q^n)$, $q>2$, such that the $\Fq$-subline determined by any three points of $S$ is entirely contained in $S$, defines an $\F_{q^k}$-subline of $\PG(1,q^n)$ for some $k|n$.
\end{result}

\begin{corollary}\label{external-Fqk-subline}
An $\Fqk$-subline $m$ in $\PG(2,q^n)$, $n$ a prime power, $q\geq n>2$, external to $l_\infty$, such that the smallest subline containing $m$ and tangent to $l_\infty$ is an $\F_{q^r}$-subline, corresponds in the ABB-representation to
 a set $M$ of $q^k+1$ affine points of $\PG(2n,q)$ if and only if
\begin{itemize}
\item[(i)] the point set $M$ spans an $r$-space intersecting $H_\infty$ in an element of $\D_r$,
\item[(ii)] every three points of $M$ define a normal rational curve $\C$ of degree $r$ in $\PG(2n,q)$ whose points are all contained in $M$, such that its $\Fqn$-extension $\C^*$ intersects the indicator set
$\{\Pi, \Pi^{\sigma}, \ldots, \Pi^{\sigma^{r-1}}\}$ of $\D_r$ in $r$ conjugate points.
\end{itemize}
\end{corollary}
\begin{proof} 
Let $m'$ be the $\F_{q^r}$-subline containing $m$ and tangent to $\l_\infty$. Any three points of $m$ define an $\F_q$-subline $m_0$ completely contained in $m$. We claim that, because $n$ is a prime power, any subline containing $m_0$ is either contained in $m$ or contains $m$. Hence, the smallest subline containing $m_0$ and tangent to $l_\infty$ equals $m'$. This holds for the following reason: consider $\PG(1,q^n)$ where $m_0$ and $m$ are the canonical $\F_q$- and $\F_{q^k}$-subline respectively, hence $m_0$ is contained in $m$. Take the point $(1,a)_{\F_{q^n}}$, $a\in\F_{q^n}$ for which $\F_{q^r}=\F_{q^k}(a)$. The subline containing $m$ and the point $(1,a)_{\Fqn}$ is then a $\F_{q^r}$-subline, say $m'$. The smallest subline containing $m_0$ and the point $(1,a)_{\Fqn}$ also equals $m'$, since $\F_{q}(a)=\F_{q^k}(a)=\F_{q^r}$, because $n$ is a prime power and $a\notin \F_{q^k}$.

The statement then follows immediately from Theorem \ref{externalsubline}.
\end{proof}

\subsection{Sublines contained in $l_{\infty}$}\label{Sectionsecantsublines}

\begin{theorem}\label{secantsubline}
Let $\S$ be a set of $q^k+1$ elements of the Desarguesian spread $\D$ of $H_\infty\simeq\PG(2n-1,q)$, $q>2$. Then the following statements are equivalent:
\begin{itemize}
\item[{\rm(1)}]
$\S$ is the image in the ABB-representation of an $\F_{q^k}$-subline of $l_\infty$,
\item[{\rm(2)}]
for any three elements of $\S$, the unique regulus through them is contained in $\S$,
\item[{\rm(3)}]
there exists a $(2k-1)$-dimensional subspace of $H_{\infty}$ intersecting each element of $\S$ in a $(k-1)$-dimensional space,
\item[{\rm(4)}]
there exists a $(2k-1)$-dimensional subspace of $H_{\infty}$ intersecting each element of $\S$ in a $(k-1)$-dimensional space of $\D_k$.
\end{itemize}
\end{theorem}
\begin{proof}
$(1)\iff(2)$\\
The ABB-map $\phi$, when restricted to the points of the line $l_\infty$, clearly corresponds to applying field reduction to the points of $\PG(1,q^n)$. By field reduction, an $\Fq$-subline contained in $l_\infty$ corresponds to a regulus of $\D$, and vice versa.

A set $\S$ of $q^k+1$ elements of $\D$ is the image of a set $m$ of $q^k+1$ points of $l_\infty$. By the previous paragraph, we get that, for any three elements of $\S$, the unique regulus through them is contained in $\S$ if and only if for any three points of $m$ the unique $\Fq$-subline through them is contained in $m$. Because of Result \ref{moeilijk} this is true if and only if $m$ is an $\Fqk$-subline contained in $l_\infty$.

$(3)\Rightarrow(1)$ and $(3)\Rightarrow(4)$\\
Consider a $(2k-1)$-dimensional subspace $\pi$ of $H_{\infty}$ intersecting each of the $q^k+1$ elements of $\S$ in a $(k-1)$-space. We see that $\S \cap \pi$ is a $(k-1)$-spread of $\pi$. Take a $2k$-space $\Pi$ of $\PG(2n,q)$ intersecting $H_\infty$ in $\pi$. Clearly, the ABB-representation $\overline{A(\S\cap\pi)}$ contained in $\Pi$ is a projective plane of order $q^k$. This subplane is embedded in the original plane $\PG(2,q^n)$. Every $\Fqk$-subline of this plane is contained in a line of $\PG(2,q^n)$, hence the subplane is a subgeometry isomorphic to $\PG(2,q^k)$. It follows that $\S$ is the image of an $\Fqk$-subline contained in $l_\infty$.

Moreover, since every tangent $\Fqk$-subline of this subplane corresponds to a $k$-space intersecting $H_\infty$ in an element of $\D_k$ (Theorem \ref{tangentsubline}), we know that all $(k-1)$-spaces of $\S \cap \pi$ belong to $\D_k$.

$(1)\Rightarrow(4)$\\
Suppose $\S$ is the image of an $\Fqk$-subline $m$ contained in $l_\infty$.
The subline $m$ together with an $\Fqk$-subline tangent to $l_\infty$ in a point of $m$, defines a unique $\Fqk$-subplane $\mu$ of $\PG(2,q^n)$. The image of the $q^{2k}$ affine points of $\mu$ in the ABB-representation is a set $M$ of $q^{2k}$ affine points of $\PG(2n,q)$.

Every two affine points of $\mu$ are contained in an $\Fqk$-subline of $\mu$ that is tangent to $l_\infty$. Hence, by Theorem \ref{tangentsubline}, every two affine points of the set $M$ are contained in an affine $k$-space completely contained in $M$. This means that the affine line through any two points of $M$ is contained in $M$, hence $M$ is an affine subspace. Since $M$ contains $q^{2k}$ points, it is a $2k$-dimensional affine subspace. Its projective completion intersects $H_\infty$ in a $(2k-1)$-space $\pi$, and clearly this space $\pi$ can intersect $\D$ only in elements of $\S$.

Consider any affine point $P$ of $\mu$ and its image $\phi(P)$ of $M$. There are $q^k+1$ $\Fqk$-sublines of $\mu$ tangent to $l_\infty$ and containing $P$. So, there are $q^k+1$ affine $k$-spaces through $\phi(P)$ contained in $M$, and, because of Theorem \ref{tangentsubline}, their projective completion intersects $H_\infty$ in an element of $\D_k$. Hence $\pi$ intersects each element of $\S$ in an element of $\D_k$.

$(4) \Rightarrow (3)$\\
Obvious.
\end{proof}

\begin{corollary}\label{subspreads}
A Desarguesian $(n-1)$-spread of $\PG(rn-1,q)$ has a unique Desarguesian $(k-1)$-subspread for each $k|n$.
\end{corollary}
\begin{proof}
In Subsection \ref{fieldreduction} we constructed a Desarguesian $(k-1)$-subspread of a Desarguesian $(n-1)$-spread of $\PG(rn-1,q)$.

To prove that such a spread is unique, first consider the case $r=2$. Consider the Desarguesian $(n-1)$-spread $\D$ of $\PG(2n-1,q)$ and any Desarguesian $(k-1)$-subspread $\S_k$ of $\D$. Take two elements of $\S_k$ contained in different elements of $\D$. These span a $(2k-1)$-space containing $q^k+1$ elements of $\S_k$, all contained in different elements of $\D$. From the proof of Theorem \ref{secantsubline} (3) $\Rightarrow$ (4), it follows that all elements of $\S_k$ are elements of $\D_k$.

Now consider $r>2$ and a Desarguesian $(n-1)$-spread $\D'$ of $\PG(rn-1,q)$. The spread $\D'$ induces a Desarguesian spread $\D$ in any $(2n-1)$-space spanned by two elements of $\D'$. By the previous part this spread $\D$ has a unique Desarguesian $(k-1)$-subspread, and the statement follows.
\end{proof}

\section{The ABB-representation of subplanes}\label{subplanes}

An $\F_{q^k}$-subplane is said to be {\it secant}, {\it tangent}, or {\it external} if it is secant, tangent, or external to $l_{\infty}$. In this section we will characterise secant $\Fqk$-subplanes and tangent $\Fq$-subplanes.

\subsection{Secant subplanes}\label{Sectionsecantsubplanes}

\begin{theorem}\label{secantsubplane}
A set $\Pi$ of affine points in $\PG(2n,q)$ is the ABB-representation of the affine points of an $\F_{q^k}$-subplane in $\PG(2,q^n)$ secant to $l_\infty$ if and only if
\begin{itemize}
\item[{\rm (i)}] $\Pi$ is a $2k$-dimensional affine space,
\item[{\rm (ii)}] its projective completion $\overline{\Pi}$ intersects $H_\infty$ in a $(2k-1)$-space which intersects $q^k+1$ elements of $\D$ in exactly a $(k-1)$-space.
\end{itemize}
Moreover, this $(2k-1)$-space intersects each of the $q^k+1$ spread elements of $\D$ in a $(k-1)$-space of $\D_k$.
\end{theorem}
\begin{proof}
Follows from the proof of Theorem \ref{secantsubline}.
\end{proof}

\subsection{Equivalent subplanes under the stabiliser of $l_\infty$}
An $\Fqk$-subplane of $\PG(2,q^n)$ is uniquely determined by four projective points in general position, or alternatively, by three independent vectors in $V(3,q^n)$.
Given $\omega,\lambda \in \Fqn$, denote by  $\pi_{\omega,\lambda}$ the $\Fq$-subplane determined by the vectors $(1,0,\lambda),(0,1,\omega),(0,0,1)$, i.e.
\begin{align*}
\pi_{\omega,\lambda} &= \{ (s,u,s\lambda+u\omega+t)_{\Fqn}\mid s,t,u \in \Fq \}\\
&=\{ (s,1,s\lambda+\omega+t)_{\Fqn}\mid s,t \in \Fq \} \cup \{(1,0,\lambda+t)_{\Fqn}\mid t\in \Fq\}\cup\{(0,0,1)_{\Fqn}\}.
\end{align*}
We see that the plane $\pi_{\omega,\lambda}$ is an external subplane if and only if $\{1,\lambda,\omega\}$ are linearly independent over $\Fq$.

\begin{lemma}\label{equivalentsubplanes}
Every $\F_{q}$-subplane external to $l_{\infty}$ is equivalent under the action of the stabiliser of $l_{\infty}$ in $\PGammaL(3,q^n)$ to $\pi_{\omega,\lambda}$ for some $\omega,\lambda \in \Fqn$ such that $\{1,\omega,\lambda\}$ are linearly independent over $\Fq$.

Every $\F_{q}$-subplane tangent to $l_{\infty}$ is equivalent under the action of the stabiliser of $l_{\infty}$ in $\PGammaL(3,q^n)$ to $\pi_{\omega,0}$ where $\omega \notin \Fq$.
\end{lemma}
\begin{proof}
Consider an $\F_{q}$-subplane $\mu$. Then $\mu$ is determined by three distinct independent vectors in $\Fqn^3$, which we may take to be $(\alpha,\beta,\omega)$, $(\gamma, \delta, \lambda)$ and $(\epsilon,\zeta,1)$. Then the matrix
\[
\label{eqn:stabilizermatrix}
\npmatrix{\gamma-\lambda \epsilon&\delta - \lambda \zeta&0\\ \alpha -\omega \epsilon & \beta - \omega \zeta&0\\\epsilon&\zeta&1}
\]
is non-singular and maps $(0,1,\omega)$ to $(\alpha, \beta, \omega)$, $(1,0,\lambda)$ to $(\gamma, \delta, \lambda)$ and $(0,0,1)$ to $(\epsilon, \zeta, 1)$, and hence the corresponding collineation $\psi \in \PGL(3,q^n)$ maps $\pi_{\omega,\lambda}$ to $\mu$.

As $\psi$ stabilises $\l_{\infty}$, if $\mu$ is disjoint from, tangent to, or external to $l_{\infty}$, then $\pi_{\omega,\lambda}$ is also disjoint from, tangent to, or external to $l_{\infty}$. If the subplane $\mu$ is tangent to $l_\infty$, we may choose our first vector such that $\lambda=0$, and hence $\mu$ is equivalent to the subplane $\pi_{\omega,0}$. If $\omega \in \Fq$, clearly the subplane is secant to $l_\infty$.
\end{proof}

\begin{lemma}\label{smallest2}
If $\Fqk = \Fq(\omega)$ is the smallest subfield of $\Fqn$ containing $\omega$, then the smallest subplane secant to $l_{\infty}$ and containing $\pi_{\omega,0}$ is an $\Fqk$-subplane.
\end{lemma}
\begin{proof}
For all $k|n$, there is a unique $\F_{q^k}$-subplane containing the points $(1,0,0)_{\Fqn},(0,1,\omega)_{\Fqn}$, $(0,0,1)_{\Fqn}$, $(1,1,\omega+1)_{\Fqn}$, and every such subplane contains the points of $\pi_{\omega,0}$. Such an $\F_{q^k}$-subplane is secant to $l_\infty$ if and only if it contains the point $(0,1,0)_{\Fqn} \in l_\infty$, if and only if $-\omega\in \F_{q^k}$, and the statement follows.
\end{proof}

\subsection{Tangent subplanes}

Consider two normal rational curves $\C_1$ and $\C_2$ of degree $k$ and $l$ respectively. Embed both curves in $\PG(N,q)$, $N\geq k+l+1$, such that the subspaces they span, of dimension $k$ and $l$ respectively, are disjoint.


Let $\rho_1,\rho_2$ be maps from $\PG(1,q) \rightarrow \PG(N,q)$ defined by $\rho_1:(s,t)_{\Fq} \mapsto \left( \sum_{i=0}^k s^{k-i} t^i e_i \right)_{\Fq}$, $\rho_2:(s,t)_{\Fq} \mapsto \left( \sum_{i=0}^l s^{l-i} t^i f_i \right)_{\Fq}$, for some vectors $e_i,f_i$ such that $\C_1 =\text{Im}(\rho_1)$ and $\C_2 = \text{Im}(\rho_2)$.

A {\em normal rational scroll of bidegree $\{k,l\}$} defined by $\C_1$, $\C_2$ consists of the set of lines of $\PG(N,q)$ defined as follows \[\left\{\left\langle \rho_1(P) , \rho_2(\psi(P)) \right\rangle \mid P \in \PG(1,q)\right\},\]
where $\psi$ is an element of $\PGL(2,q)$.

The stabiliser in $\PGL(N+1,q)$ of a normal rational curve in $\PG(N,q)$ is isomorphic to $\PGL(2,q)$ \cite[Theorem 27.5.3]{GGG}. So, if we take different choices $\rho_i'$ defining $\C_i$, then $\rho_i' = \rho_i \psi_i$ for some $\psi_i \in \PGL(2,q)$, and the normal rational scroll defined by $\rho_1$, $\rho_2$ and $\psi$ is equal to the one defined by $\rho_1'$, $\rho_2'$, and $\psi_2^{-1} \psi \psi_1$. Hence the set of all normal rational scrolls defined by $\C_1$, $\C_2$ does not depend on the choice of $\rho_i$.

When $\C_1$ and $\C_2$ are general curves, such a set of lines is often called a {\em ruled surface}. The curve of the smallest degree is often called the {\em directrix}.

Consider the embedding of $\PG(N,q)$ as a subgeometry of $\PG(N,q^n)$. The $\Fqn$-extensions of the curves $\C_i$ are unique and we can consider the canonical extension $\rho_i^*:\PG(1,q^n)\rightarrow\PG(N,q^n)$ of $\rho_i$ and $\psi^* \in \PGL(2,q^n)$ of $\psi$. Clearly these define the \emph{$\F_{q^n}$-extension} of a normal rational scroll and such an extension is unique.

Before considering the characterisation of external subplanes, we introduce a lemma on the existence of normal rational scrolls.

\begin{lemma}\label{uniquescroll}
Let $\C_1, \C_2$ be two normal rational curves in $\Sigma$, and $\C_1^*,\C_2^*$ their respective extensions to $\Sigma^*$. Consider the points $P \in \C_1^*\backslash\C_1$, $Q \in \C_2^*\backslash\C_2$ such that
$P$ is not  contained in the $\F_{q^2}$-extension of $\C_1$. 
Then there exists at most one normal rational scroll $S^*$ containing $\C_1^*$ and $\C_2^*$ such that $S^*$ contains the line $\langle P, Q\rangle$, and such that $S^*$ meets $\Sigma$ in a normal rational scroll $S$ containing $\C_1$ and $\C_2$. Moreover, if such a scroll exists, then it contains the lines $\langle P^{\sigma^i}, Q^{\sigma^i} \rangle$ for each $i$, where $\Sigma = \mathrm{Fix}(\sigma)$.
\end{lemma}

\begin{proof}
Let $\C_i^*$ be defined by the map $\rho_i^*$ from $\PG(1,q^n)$ into $\Sigma^*$ such that $\C_i$ is the image under $\rho_i^*$ of the $\Fq$-subline  $l = \{(s,t)_{\Fqn} \mid s,t \in \Fq, (s,t) \ne (0,0)\} \simeq \PG(1,q)$ of $\PG(1,q^n)$.
Now $P = \rho_1^*((s,t)_{\Fqn})$ and $Q = \rho_2^*((u,v)_{\Fqn})$ for some $s,t,u,v \in \Fqn \backslash \Fq$. Note that $s/t \notin \F_{q^2}$, for otherwise $P$ would be contained in the $\F_{q^2}$-extension of $\C_1$.

Hence a normal rational scroll exists if and only if there exists a $\psi \in \PGL(2,q^n)$ fixing $l$ (whence $\psi \in \PGL(2,q)$) and mapping  $(s,t)_{\Fqn}$ to $(u,v)_{\Fqn}$. The stabiliser of the point $(s,t)_{\Fqn}$ under the action of $\PGL(2,q)$ is trivial unless $s/t \in \F_{q^2}$, and hence the result follows.

Now $P^{\sigma} = \rho_1^*((s,t)_{\Fqn})^{\sigma} = \rho_1^*((s^{\sigma},t^{\sigma})_{\Fqn})$, and similarly $Q^{\sigma} =  \rho_2^*((u^{\sigma},v^{\sigma})_{\Fqn})$. As $\psi \in \PGL(2,q)$, we have that $\psi((s^{\sigma},t^{\sigma})_{\Fqn}) = (u^{\sigma},v^{\sigma})_{\Fqn}$, and so this scroll contains $\langle P^{\sigma^i}, Q^{\sigma^i} \rangle$ for each $i$.
\end{proof}


It has been shown (\cite{BarJack2} for $n=3$ and \cite{LaZa201?} for general $n$) that the image of a tangent subplane is a normal rational scroll, and its bidegree calculated. We are now ready to fully characterise tangent subplanes in $\PG(2,q^n)$, extending the result of \cite{BarJack2} for $n=3$ to general $n$.

\begin{theorem}\label{tangentsubplane}
A set $S$ of affine points of $\PG(2n,q)$, $q\geq n$, corresponds to the affine points of a tangent subplane $\mu$ if and only if
$S$ consists of the affine points of a normal rational scroll defined by curves $\C,\N$ satisfying:
    \begin{itemize}
    \item[{\rm(S1)}] $\C$ is a normal rational curve of degree $k$ contained in an affine $k$-space $\pi$, for which $\overline{\pi}\cap H_\infty$ is an element $E_1$ of $\D_k$, such that its $\Fqn$-extension $\C^*$ contains all conjugate points $\{P,P^{\sigma},\ldots,P^{\sigma^{k-1}}\}$ generating the spread element $E_1$,
    \item[{\rm (S2)}]$\N$ is a normal rational curve of degree $k-1$ contained in an element $E_2$ of $\D_k$, where $E_1$ and $E_2$ are not contained in the same element of $\D$, such that its $\Fqn$-extension $\N^*$ contains all conjugate points $\{Q,Q^{\sigma},\ldots,Q^{\sigma^{k-1}}\}$ generating the spread element $E_2$,
    \item[{\rm(S3)}] the $\Fqn$-extension of the normal rational scroll contains the lines $\langle P^{\sigma^{i}}, Q^{\sigma^{i}}\rangle$, each line contained in an indicator space $\Pi^{\sigma^i}$ of $\D_k$.
    \end{itemize}
for some $k|n$.

Moreover, in that case the smallest subplane containing $\mu$ and secant to $l_\infty$ is an $\Fqk$-subplane.

%
\end{theorem}

\begin{proof}
First, suppose the smallest secant subplane containing $\mu$ is an $\F_{q^2}$-subplane. From Theorem \ref{secantsubplane} the affine points of $\mu$ are contained in a 4-dimensional affine space intersecting $H_\infty$ in a $3$-space partioned by lines of $\D_1$. In this case we can use the characterisation of the ABB-representation in $\PG(4,q)$ of a tangent Baer subplane of $\PG(2,q^2)$ considered in \cite[Theorem 3.1.9]{Ebert}. This corresponds to a normal rational scroll satisfying $(S1),(S2),(S3)$ with $k=2$, proving our claim. Note that the normal rational curve $\N$ of degree 1 is just a projective line of $\D_1$.


We can now consider $k>2$.
From Lemma \ref{equivalentsubplanes} and \ref{smallest2} the tangent $\F_{q}$-subplane $\mu$ is equivalent to $\pi_{\omega,0}$, where $\Fq(\omega)=\Fqk$, under an element of the stabiliser $G$ of $l_\infty$ in $\PGL(3,q^n)$, say $\chi_{{}_0}(\pi_{\omega,0})=\mu$ for $\chi_{{}_0}\in G$. Note that $\pi_{\omega,0}$ consists of the following points:
\[
\pi_{\omega,0} 
=\left\{ \left(\frac{s}{\omega+t},\frac{1}{\omega+t},1\right)_{\Fqn}\middle| \ s,t \in \Fq \right\} \cup \{(t,0,1)_{\Fqn}\mid t\in \Fq\}\cup\{(1,0,0)_{\Fqn}\}.
\]
In the ABB-representation the point $(1,0,0)_{\Fqn} \in \pi_{\omega,0}$ corresponds to the spread element $\{ (x,0,0)_{\Fq} \mid x \in \F_{q^n}^*\} \in \D$.
By definition of the ABB-map $\phi$, the affine points of $\pi_{\omega,0}$ in the ABB-representation form a set $W$ defined as follows:
\begin{align*}
W &=\left\{\left(\frac{s}{\omega+t},\frac{1}{\omega+t},1\right)_{\Fq} \middle| \ s,t \in \Fq\right\}\cup\left\{\left(t,0,1\right)_{\Fq} \mid t\in\Fq\right\}.
\end{align*}
This is a set of $q+1$ affine lines; one can see this more clearly when we consider the set $\overline{W}$ consisting of the projective completions of all lines of $W$.
\begin{align*}
\overline{W}  &= \left\{\left\langle \left(\frac{1}{\omega+t},0,0\right)_{\Fq}, \left(0,\frac{1}{\omega+t},1\right)_{\Fq}\right\rangle\middle| \ t\in\Fq\right\} \cup \left\langle \left(1,0,0\right)_{\Fq}, \left(0,0,1\right)_{\Fq}\right\rangle
\end{align*}
From the proof of Theorem \ref{externalsubline} we know that the set
\[
\C_{\omega} = \left\{ \left(0,\frac{1}{\omega+ t},1\right)_{\Fq} \middle| \ t \in \Fq \right\}\cup \left\{(0,0,1)_{\Fq}\right\}.
\]
is a normal rational curve satisfying the conditions in (S1).

Consider the following set of points of $H_\infty$, all contained in $\overline{W}$:
\begin{align*}
\N_{\omega} &= \left\{ \left(\frac{1}{\omega+ t},0,0\right)_{\Fq} \middle| \ t \in \Fq \right\}\cup \left\{(1,0,0)_{\Fq}\right\}\\
&= \left\{ \left(\prod_{i=1}^{k-1} (\omega^{q^i}+t),0,0\right)_{\Fq} \middle| \ t \in \Fq \right\}\cup \left\{(1,0,0)_{\Fq}\right\}.
\end{align*}
A similar calculation shows that $\N_{\omega}$ is a normal rational curve of degree $k-1$, contained in a projective $(k-1)$-space $E_2$ of $\D_k$, as claimed. In fact, this has been shown in e.g. \cite{InversionMap1,InversionMap2}.
Note that the spread element of $\D$ in which $E_2$ is contained is not equal to the one associated to $\C_{\omega}$.

Now using the embedding $\iota$ to embed $\N_{\omega}$ in $\PG(2n,q^n)$ gives the points
\[
 \left(1,\ldots,1;0,\ldots,0;0\right)_{\Fqn},
\]
and
\[
 \left(\prod_{i=1}^{k-1} (\omega^{q^i}+t),\prod_{i=1}^{k-1} (\omega^{q^{i+1}}+t),\ldots,\prod_{i=1}^{k-1} (\omega^{q^{i+n-1}}+t);0,0,\ldots,0;0\right)_{\Fqn},
\]
for $t \in \Fq$.  Since $\F_q(\omega)=\F_{q^k}$, the $(n+j)$-th entry is equal to the $(n+j+k)$-th entry for all $0\leq j \leq n-1$. The extension $\N^*_{\omega}$ of $\N_\omega$ is the normal rational curve obtained by allowing $t$ to range over $\Fqn$. As before, one can see that $\N^*_{\omega}$ contains the conjugate points of the indicator set of $\D_k$ generating the $(k-1)$-space in which $\N_{\omega}$ lies, and hence $\N_{\omega}$ satisfies the condition of (S2).

Clearly then the points of $W$ are precisely the affine points of the normal rational scroll $\B_{\omega}$ defined by the curves $\C_{\omega}$, $\N_{\omega}$ and the identity element of $\PGL(2,q)$, and $\B_{\omega}$ satisfies (S3).

To prove that the converse also holds, it is again sufficient to use a counting argument.

We know from Theorem \ref{externalsubline} that sublines $l_0$ external to $l_{\infty}$ are in one-to-one correspondence with normal rational curves satisfying (S1). The number of tangent $\Fq$-subplanes containing a fixed external $\Fq$-subline and a fixed point $R$ of $l_{\infty}$ (not on the extension of this fixed external subline) is $\frac{q^n-1}{q-1}$. 

Therefore it suffices to show that for a fixed $\C$ satisfying (S1) and a fixed element $E$ of $\D$ (for which $\C$ does not lie in a $n$-space containing $E$), there are precisely $\frac{q^n-1}{q-1}$ curves $\N$ contained in $E$ satisfying (S2) for which there exist a normal rational scroll $\B$ defined by $\C$ and $\N$ and satisfying (S3).

Using the induced action of $G$ (the stabiliser of the line $l_\infty$ in $\PGL(3,q^n)$) on $\PG(2n,q^n)$ we can choose w.l.o.g. $\C = \C_{\omega}$, where $\Fq(\omega)=\Fqk$, with $k>2$ by our assumption, and $E = \iota\phi(1,0,0)_{\Fqn}$. Consider a point $A = \iota(a,0,0)_{\Fq}$ of $E$ and the element $E_2$ of $\D_k$ containing $A$. We will prove that there are $q(q+1)$ choices for a point $B \in E_2$ such that $A,B$ and the points of $\C_{\omega}$ lie on a normal rational scroll satisfying (S1),(S2),(S3). We let $B = \iota(b,0,0)_{\Fq}$ for some $b \in \Fqn$. The conjugate points generating $E_2$ are $\{R,R^{\sigma},\ldots,R^{\sigma^{k-1}}\}$, where $R=(v)_{\Fqn}$ with
\begin{align*}
v &= (1,0,\ldots,0,a^{q^k-1},0,\ldots,0,a^{q^{n-k}-1},0,\ldots,0;0,\ldots,0;0)\\
 &= (1,0,\ldots,0,b^{q^k-1},0,\ldots,0,b^{q^{n-k}-1},0,\ldots,0;0,\ldots,0;0).
\end{align*}

Now there exists a unique normal rational curve $\N_{AB}$ satisfying (S2), i.e. such that its extension contains $\{A,B,R,R^{\sigma},\ldots,R^{\sigma^{k-1}}\}$. We can see that it must be the curve defined by the following map:
\[
\eta(s,t) = \left(\sum_{i=0}^{k-1} \prod_{j=0,j \ne i}^{k-1} (a^{q^j} s - b^{q^j}t) v^{\sigma^i}\right)_{\Fqn}.
\]
Then $R^{\sigma^i} = \eta(b^{q^i},a^{q^i})$, $A = \eta(0,1)$, $B = \eta(1,0)$, and $\N_{AB} = \eta(\Fq \times \Fq)$.

From the proof of Theorem \ref{externalsubline}, the curve $\C^*_{\omega}$ can be parametrised by a map $\rho: \Fqn \times
\Fqn \rightarrow \PG(2n,q)$ such that $\C_{\omega} = \rho(\Fq \times \Fq)$, and the intersection of $\C^*_{\omega}$ with the indicator sets are the points $Q^{\sigma^i} = \rho(1,-\omega^{q^i})$, $i \in \{0,\ldots,k-1\}$. Using Lemma \ref{uniquescroll}, there exists a normal rational scroll defined by $\C_{\omega}$ and $\N_{AB}$ 
such that its $\Fqn$-extension contains the lines $\langle Q^{\sigma^i}, R^{\sigma^i}\rangle$ (hence satisfies (S3)) if and only if there exists an element $\psi$ of $\PGL(2,q^n)$ which fixes the canonical $\Fq$-subline (defined by $\Fq \times \Fq$), whence $\psi \in \PGL(2,q)$), and which maps $(1,-\omega)_{\Fqn}$ to $(1,\frac{b}{a})_{\Fqn}$. By the proof of Lemma \ref{uniquescroll}, since $\omega \notin \Fqt$, there are $|\PGL(2,q)|=q(q^2-1)$ points in the orbit of $(1,-\omega)_{\Fqn}$ under $\PGL(2,q)$, and hence $q(q^2-1)$ choices for $b$. As $\iota(b,0,0)_{\Fq} = \iota(\lambda b,0,0)_{\Fq}$ for all $\lambda \in \Fq^{\times}$, we get that there are $q(q+1)$ allowable points $B \in E_2$.
Since there are $\frac{q^n-1}{q-1}$ choices for $A$, and each $\N^*_{AB}$ contains $q(q+1)$ ordered pairs of distinct points, we have that there are precisely $\frac{q^n-1}{q-1}$ curves satisfying (S2) which define a scroll satisfying (S3), proving the claim.
\end{proof}

\begin{remark}
In \cite{BarJack2}, in the case $n=3$ a curve satisfying (S1) was referred to as a \emph{special normal rational curve}, while a curve satisfying (S2) was referred to as a \emph{special conic}.
\end{remark}

\begin{remark} Proceeding in the same way as for tangent subplanes, we can describe the ABB-representation of external $\F_q$-subplanes. The ABB-representation of $\pi_{\omega,\lambda}$, where $\{1,\omega, \lambda\}$ are linearly independent over $\F_q$, is the set $S$ of points of the form \[
\left(\prod_{i=1}^{r-1} s(s\lambda^{q^i} + u\omega^{q^i}+t),\prod_{i=1}^{r-1} u(s\lambda^{q^i} + u\omega^{q^i}+t),\prod_{i=0}^{r-1} (s\lambda^{q^i} + u\omega^{q^i}+t)\right)_{\Fq},
\]
for $s,t,u \in \Fq$. Expanding the brackets, we find ${n+2 \choose 2}$ vectors $v_{ijm} \in \Fqn \times \Fqn \times \Fq$ such that
\[
S = \left\{ \left(\sum_{i+j+m=n}  s^i t^j u^m v_{ijm} \right)_{\Fq} \middle| \ s,t,u \in \Fqn\right\}.
\]
Hence $S$ is a projection of the Veronese variety $V_x(x) \subset \PG({n+2 \choose 2}-1,q)$ into $\PG(2n,q)$.

By Theorem \ref{externalsubline}, we know that this projection of the Veronese variety is covered by normal rational curves: through every two points of $S$, there lies a unique normal rational curve, entirely contained in $S$.
However, the characterisation of these projected Veronese varieties that correspond to external subplanes remains an open problem.
\end{remark}

{\bf Acknowledgment:} The authors would like to thank Michel Lavrauw for his helpful remarks during the authors' research visit to Universit\`a degli Studi di Padova in Vicenza.

%
%

\noindent
Affiliation of the authors:\\

{S. Rottey}\\
srottey\makeatletter @vub.ac.be\\
Department of Mathematics, Vrije Universiteit Brussel\\
Pleinlaan 2, 1050 Brussel, Belgium\\

{J. Sheekey}\\
jsheekey\makeatletter @cage.ugent.be\\
Department of Mathematics, Universiteit Gent\\
Krijgslaan 281, S22, 9000 Gent, Belgium\\

{G. Van de Voorde}\\
gvdvoorde\makeatletter @cage.ugent.be\\
Department of Mathematics, Universiteit Gent\\
Krijgslaan 281, S22, 9000 Gent, Belgium\\ 

\end{document}